\documentclass[11pt]{amsart}

\usepackage{amsfonts,amsmath,amssymb,amsthm,amscd,amsxtra}
\usepackage{enumerate,verbatim}
\usepackage[dvipsnames]{xcolor}
\usepackage[all,2cell,ps]{xy}
\usepackage[pagebackref]{hyperref}
\usepackage{mathptmx}
\usepackage[normalem]{ulem}
\usepackage{soul}
\theoremstyle{plain} 
\newtheorem{thm}{Theorem}[section]
\newtheorem{prop}[thm]{Proposition}
\newtheorem{cor}[thm]{Corollary}

\newtheorem{lem}[thm]{Lemma}

\theoremstyle{definition}

\newtheorem{eg}[thm]{Example}

\newtheorem{ques}[thm]{Question}

\newtheorem{rmk}[thm]{Remark}

\newtheorem{chunk}[thm]{}

\numberwithin{equation}{section}

\newcommand{\lra}{\longrightarrow}
\newcommand{\lla}{\longleftarrow}

\newcommand{\xla}{\xleftarrow}\newcommand{\la}{\leftarrow}
\newcommand{\xra}{\xrightarrow}

\newcommand{\fm}{\mathfrak{m}}
\newcommand{\fn}{\mathfrak{n}}
\newcommand{\fp}{\mathfrak{p}}
\newcommand{\fq}{\mathfrak{q}}
\newcommand{\up}[1]{\parbox[][1em][]{0pt}{}^{#1}\!}

\DeclareMathOperator{\Ann}{Ann}
\DeclareMathOperator{\Jac}{Jac}
\DeclareMathOperator{\Ass}{Ass}

\DeclareMathOperator{\e}{e}
\DeclareMathOperator{\E}{E}

\DeclareMathOperator{\Tr}{Tr}
\DeclareMathOperator{\gr}{grade}

\newcommand{\crs}{\operatorname{crs}}
\newcommand{\crsp}{\crs_p}

\newcommand{\drs}{\operatorname{drs}}
\newcommand{\drsp}{\drs_p}

\DeclareMathOperator{\coker}{coker}

\DeclareMathOperator{\depth}{depth}

\DeclareMathOperator{\Ext}{Ext}

\DeclareMathOperator{\length}{length}

\DeclareMathOperator{\Hom}{Hom}

\DeclareMathOperator{\id}{id}
\DeclareMathOperator{\im}{im}

\DeclareMathOperator{\pd}{pd}
\DeclareMathOperator{\fd}{fd}

\DeclareMathOperator{\rank}{rank}

\DeclareMathOperator{\Soc}{Soc}
\DeclareMathOperator{\Supp}{Supp}
\DeclareMathOperator{\Spec}{Spec}

\DeclareMathOperator{\Tor}{Tor}

\DeclareMathOperator{\Min}{Min}
\DeclareMathOperator{\Max}{Max}

\newcommand{\ov}{\overline}
\newcommand{\ol}{\overline}

\newcommand{\wh}{\widehat}
\newcommand{\x}{\underline{\mathbf{x}}}
\newcommand{\F}{\mathbf{F}}
\newcommand{\I}{\mathbf{I}}
\renewcommand{\ge}{\geqslant} \renewcommand{\le}{\leqslant} 
\renewcommand{\geq}{\geqslant} \renewcommand{\leq}{\leqslant} 
\renewcommand{\phi}{\varphi}

\def\urltilda{\kern -.15em\lower .7ex\hbox{\~{}}\kern .04em}
\def\urldot{\kern -.10em.\kern -.10em}\def\urlhttp{http\kern -.10em\lower -.1ex
\hbox{:}\kern -.12em\lower 0ex\hbox{/}\kern -.18em\lower 0ex\hbox{/}}

\begin{document}

\title[On the test properties of the Frobenius Endomorphism]{On the test properties of the Frobenius Endomorphism}

\author[Olgur Celikbas]{Olgur Celikbas}
\address{Olgur Celikbas\\School of Mathematical and Data Sciences\\
West Virginia University}
\email{olgur.celikbas@math.wvu.edu}

\author[Arash Sadeghi]{Arash Sadeghi}
\address{Arash Sadeghi\\ School of Mathematics, Institute for Research in Fundamental Sciences (IPM), P.O. Box: 19395-5746, Tehran, Iran}
\curraddr{Dokhaniat 49179-66686, Gorgan, IRAN}	
\email{sadeghiarash61@gmail.com}

\author[Yongwei Yao]{Yongwei Yao}
\address{Yongwei Yao\\
Department of Mathematics and Statistics, 
Georgia State University, 
Atlanta, GA 30303, USA}
\email{yyao@gsu.edu}



\begin{abstract} In this paper, we prove two theorems concerning the test properties of the Frobenius endomorphism over commutative Noetherian local rings of prime characteristic $p$. Our first theorem generalizes a result of Funk-Marley on the vanishing of Ext and Tor modules, while our second theorem generalizes one of our previous results on maximal Cohen-Macaulay tensor products. In these earlier results, we replace $\up{e}{R}$ with a more general module $\up{e}{M}$, where $R$ is a Cohen-Macaulay ring, $M$ is a Cohen-Macaulay $R$-module with full support, and $\up{e}{M}$ is the module viewed as an $R$-module via the $e$-th iteration of the Frobenius endomorphism. We also provide examples and present applications of our results, yielding new characterizations of the regularity of local rings.
\end{abstract}

\maketitle

\section{Introduction}\label{sec:intro}

Throughout the paper, all rings are assumed to be commutative and Noetherian. By $(R, \fm, k)$, we mean that $R$ is a local ring with a unique maximal ideal $\fm$ and residue field $k$. 

Let $R$ be a ring of prime characteristic $p$, $F:R\to R$ be the Frobenius endomorphism, and let $M$ be an $R$-module. Each iteration $F^e$ of $F$ defines a new $R$-module structure on $M$, denoted by $\up{e}{M}$, whose scalar multiplication is given as follows: For $r\in R$ and $x\in \up{e}{M}$, we have that $r \cdot x=r^{p^e}x$.  We say that $R$ is \emph{F-finite} if, for some $e\geq 1$ (or equivalently, for all $e\geq 1$), the module $\up{e}R$ is finitely generated over $R$; see, for example, \cite{PS4}. We denote by $F^e_R(-)$ the scalar extension along the $e$th iteration $F^e_R: R \to R$ of $F$. Thus, if $\sum_{i} n_i \otimes s_i\in F^e_R(M)$, where $F^e_R(M)=M\otimes_R\up{e}R$, $n_i\in M$ and $r_i \in R$, then $r \cdot \big(\sum_{i} n_i \otimes s_i \big)= \sum_{i} n_i \otimes (rs_i)$, with $rs_i$ being the product of $r$ and $s_i$ in $R$. Note that $F_R^e(M)$ is the $S$-module $M\otimes_RS$ obtained via the base change $F^e:R\to S=R$.

The module structure of $\up{e}R$ (as an $R$-module) contains important information about the homological properties of the ring $R$.  For example, a remarkable result of Kunz \cite{KunzRegular} shows that $R$ is regular if and only if $\up{e}R$ is a flat $R$-module for some (or equivalently, for all) $e \ge 1$; see also \cite{AHIY} and \cite{Rodio} for extensions of this result. Motivated by Kunz's result, the test properties of the Frobenius endomorphism have been extensively studied. 

If $N$ is a finitely generated $R$-module, it follows from the work of Herzog \cite{HerzogCharp} and Peskine-Szpiro \cite{PS3, PS4} that $\pd_R(N)<\infty$ if and only if $\Tor_i^R(\up{e}R, N)=0$ for infinitely many $e$ and for all $i\geq 1$. Avramov-Miller \cite{Av-Mi} showed that, if $R$ is a complete intersection, the vanishing of a single $\Tor_n^R(\up{e}R, N)$ for some $e\geq 1$ and $n\geq 1$ suffices to conclude that $\pd_R(N)<\infty$. Koh-Lee \cite{KL} developed ideas rooted in techniques of Burch \cite{Burch68}, Herzog \cite{HerzogCharp}, and Hochster \cite{Mel77}, and showed that $\up{e}R$ detects the finiteness of $N$ even when finitely many $\Tor$ modules vanish. Specifically, Koh-Lee proved that, given integers $e \gg 0$ and $t\geq 1$, if $\Tor_i^R(\up{e}R, N)=0$ for all $t, \ldots, t+\depth(R)$, then $\pd_R(N)<\infty$. They further showed that in the case where $R$ is Cohen–Macaulay, the number of vanishing $\Tor$ modules can be reduced by one. We refer the reader to the expository work \cite{MillerSurvey} of Miller for further details. 

In this paper, we focus on the following result of Funk-Marley \cite{FM,FMcorrection}, which examines the vanishing of $ \Tor_i^R(\up{e}R, N)$ for the case where $R$ is Cohen-Macaulay and $N$ is possibly an infinitely generated $R$-module.

\begin{chunk} [{Funk-Marley \cite[3.1 and 3.2]{FM}}] \label{FM-intro} Let $(R,\fm, k)$ be a $d$-dimensional Cohen-Macaulay local ring of prime characteristic $p$, with $d\geq 1$, and let $N$ be an $R$-module. Given integers $e \gg 0$ and $t\geq 1$, the following hold:  
\begin{enumerate}[\rm(i)]
\item If $\Tor_i^R(\up{e}R, N)=0$ for all $i=t,\, \dots,\, t+d-1$, then $\fd_R(N) \le d$. 
\item If $R$ is F-finite and $\Ext^i_R(\up{e}R, N)=0$ for all $i=t,\, \dots,\,t+d-1$, then $\id_R(N)\leq d$. 
\end{enumerate}
\end{chunk}

One of the main goals of this paper is to generalize the result of Funk-Marley stated in \ref{FM-intro}. In fact, we prove more and establish the following theorem:
 
\begin{thm} \label{main-thm-intro} Let $(R,\fm, k)$ be a $d$-dimensional local ring of prime characteristic $p$, $M$ be a finitely generated Cohen-Macaulay $R$-module such that $\Supp_R(M) = \Spec(R)$, and let $N$ be an $R$-module. Assume $\depth(R) \geq 1$. Given integers $t\geq 1$ and $e \gg 0$, we have the following:
\begin{enumerate}[\rm(i)]
\item If $\Tor^R_i(\up{e}M, N)=0$ for all $i=t,\, \dots,\,t+d-1$, then $\fd_R(N) \leq d$.
\item If $R$ is F-finite and $\Ext^i_R(\up{e}M, N)=0$ for all $i=t,\, \dots,\,t+d-1$, then $\id_R(N)\leq d$.
\item If $N$ is finitely generated and $\Ext^i_R(N, \up{e}M)=0$ for all $i=t,\, \dots,\,t+d-1$, then $\pd_R(N)\leq t-1$. 
\end{enumerate}
\end{thm}

Parts (i) and (ii) of Theorem \ref{main-thm-intro} recover the aforementioned result of Funk-Marley for the case where $M=R$. Note that, in each part of Theorem~\ref{main-thm-intro}, only $d$ consecutive vanishings of $\Ext$ or $\Tor$ modules are needed to conclude the homological property the module of $N$.

We first record several preliminary results in Section 2 and then prove the first two parts of Theorem~\ref{main-thm-intro} as Theorem~\ref{injective} in Section 3. The third part of Theorem~\ref{main-thm-intro} is established as Theorem ~\ref{p1} in Section 4. Additionally, in Example~\ref{ex:ext-2}, we show that the conclusion of Theorem~\ref{main-thm-intro} may fail if the ring $R$ in question has zero depth. 

Li \cite{JL2012} proved that, if $(R,\fm, k)$ is a Cohen-Macaulay local ring, $N$ is a finitely generated $R$-module with rank, and $F_R^e(N)$ is maximal Cohen-Macaulay for some $e \gg 0$, then $N$ is free. In \cite{CSY}, the authors of the present paper replaced the rank hypothesis on $N$ with the weaker assumption that $N$ is generically free, and proved the following result: 

\begin{chunk} [{Celikbas-Sadeghi-Yao \cite[1.3]{CSY}}] \label{CSY} Let $(R,\fm, k)$ be a Cohen-Macaulay local ring of prime characteristic $p$, and let $M$ and $N$ be finitely generated $R$-modules. Assume $N$ is generically free, that is, $N_{\fp}$ is free over $R_\fp$ for all $\fp \in \Ass(R)$. If $F_R^e(N)$ is maximal Cohen-Macaulay for some $e \gg 0$, then $N$ is free.  
\end{chunk}

In Section~\ref{sec:Frob-proj}, as a byproduct of Theorem~\ref{main-thm-intro}(iii), we generalize \ref{CSY} and prove the following result; see Corollary~\ref{CSY-gen}. Note that $F_R^e(N)\otimes_RM$ is the $S$-module $(N\otimes_RS)\otimes_SM$, where $R\to S=R$ is the $e$-th iteration of the Frobenius endomorphism.

\begin{thm} \label{CSY-gen-intro} Let $(R,\fm, k)$ be a local ring of prime characteristic $p$, and let $M$ and $N$ be finitely generated $R$-modules. Assume the following conditions hold:
\begin{enumerate}[\rm(i)]
\item $M$ is Cohen-Macaulay and $\Supp_R(M) = \Spec(R)$. 
\item $N$ is generically free, that is, $N_{\fp}$ is a free $R_{\fp}$-module for all $\fp \in \Ass(R)$. 
\end{enumerate}
If $F_R^e(N)\otimes_RM$ is maximal Cohen-Macaulay for some $e \gg 0$, then $N$ is free.  
\end{thm}

Examples~\ref{ornek1} and ~\ref{ornek2} showcase the necessity of the hypotheses $\Supp_R(M) = \Spec(R)$ and $N$ is generically free in Theorem~\ref{CSY-gen-intro}. An immediate consequence of Theorem~\ref{CSY-gen-intro} over one-dimensional rings can be stated as follows:

\begin{cor} Let $(R,\fm, k)$ be a one-dimensional reduced local ring of prime characteristic $p$ and let $0\neq I$ be an ideal of $R$. If $F_R^e(N)\otimes_RI$ is torsion-free for some finitely generated $R$-module $N$ and $e\gg 0$, then $N$ is free.
\end{cor}

Theorem~\ref{main-thm-intro}(iii), in addition to Theorem~\ref{CSY-gen-intro}, has other applications, namely Corollaries~\ref{cor:dual-p1}, \ref{min1}, and \ref{min2}. Moreover, in Corollary~\ref{cor:main-thm-intro}, we obtain new characterizations of the regularity in terms of the vanishing of $\Ext$ and $\Tor$.


\section{Preliminaries}

In this section, we record several preliminary results and observations that are necessary for our arguments in the subsequent sections. For the main results of this paper, one can skip this section and proceed to Section~\ref{sec:Frob-inj} and Section~\ref{sec:Frob-proj}.

\begin{chunk} \label{Begin} Let $(R,\fm, k)$ be a local ring and let $M$ be an $R$-module.
\begin{enumerate}[\rm(i)]
\item We set $M^{\vee}=\Hom_R\big(M,E_R(k)\big)$, where $E_R(k)$ is the injective hull of $k$. Note that $\Hom_R(-,E_R(k))$ is a faithful exact functor.
\item Assume that $M$ is finitely generated over $R$. Given $n\geq 1$, we denote by $\Omega^n_R M$ the $n$-th \emph{syzygy} of M, namely, the image of the $n$-th differential map in a minimal free resolution of $M$. By convention, $\Omega_R^0M=M$.
\item If $M\neq 0$ is finitely generated over $R$ and $\dim_R(M)=t$, we define the \emph{Hilbert-Samuel multiplicity} of $M$ as
$$\e_R(M)= t! \lim_{n \to \infty} \dfrac{\length_R(M/\fm^n M)}{n^{t}},$$ 
which is a positive integer; see, for example, \cite[Page 107]{Mat}. 
\end{enumerate}
\end{chunk}

\begin{chunk}  \label{DS} Let $(R,\fm, k)$ be a local ring and let $X$, $N$ be $R$-modules. 
\begin{enumerate}[\ \rm(1)]
\item Assume $\Soc_R(X)\nsubseteq \fm X$ and let $x\in \Soc_R(X) - \fm X$. This implies that $Rx \cong k$ and the nonzero map $Rx \hookrightarrow X \twoheadrightarrow X/ \fm X$, $x \mapsto \overline{x}$, splits. Therefore, $Rx \hookrightarrow X$ splits and thus $k$ is a direct summand of $X$ as an $R$-module. 
\item Assume $R$ has prime characteristic $p$ and $X=\up{e}M$ for some $R$-module $M \neq 0$ and $e \geq 0$.
\begin{enumerate}[\rm(i)]
\item If $\Soc_R(X)\nsubseteq \fm X$, that is, $\up{e}(0:_M \fm^{[p^e]}) \nsubseteq \up{e}\big(\fm^{[p^e]}M\big)$, then part (1) shows that $k$ is a direct summand of $\up{e}M$ as an $R$-module.
\item Assume that $M$ is finitely generated over $R$ with $\depth_R(M)=0$. For all $e \gg 0$, we have $\up{e}(0:_M \fm^{[p^e]})=\Soc_R(\up{e}M) \supseteq \up{e}\big(\Soc_R(M)\big) \nsubseteq \up{e}\big(\fm^{[p^e]}M\big)$. Hence, part~(i) shows that $k$ is a direct summand of $\up{e}M$ for every $e \gg 0$.
\item Assume that $(R,\fm,k)$ is Artinian. Then $\up{e}(0:_M \fm^{[p^e]}) \nsubseteq \up{e}\big(\fm^{[p^e]}M\big)$ for all $e \gg 0$. By part~(i), $k$ is a direct summand of $\up{e}M$ for all $e \gg 0$. Thus, given $i \ge 1$, if $\Ext^i_R(\up{e}M, N)=0$ for some $e \gg 0$, then $\Ext^i_R(k,N)=0$ and hence $N$ is injective; see \cite[3.1.12]{BH} and also \cite[2.0.10]{FunkThesis}.
\end{enumerate}
\end{enumerate}
\end{chunk}

\begin{chunk} \label{finite} Let $(R,\fm, k)$ be a local ring of prime characteristic $p$. 
There exists a local flat ring homomorphism  $(R,\fm,k) \to (S,\fn,\ell)$ such that $S$ is F-finite, $\fm S = \fn$ and $|\ell| = \infty$; hence $S$ is faithfully flat over $R$, $\dim(S) = \dim(R)$ and $\e(R) = \e(S)$; it also follows that $R$ is Cohen-Macaulay (respectively, regular) if and only if $S$ is Cohen-Macaulay (respectively, regular). (For example, with $\wh R \cong k[\![x_1,\dotsc,x_m]\!]/I$,  we can pick $S = \ov k[\![x_1,\dotsc,x_m]]/I\ov k[\![x_1,\dotsc,x_m]\!]$, where $\ov k$ is the algebraic closure of $k$.) In the case where $S$ is such an extension of $R$ and $M$ is a finitely generated $R$-module, it follows that $M$ is a Cohen-Macaulay $R$-module with $\Supp_R(M) = \Spec(R)$ if and only if $M \otimes_R S$ is a Cohen-Macaulay $S$-module with $\Supp_S(M \otimes_R S) = \Spec(S)$; also, $M$ is free over $R$ if and only if $M \otimes_R S$ is free over $S$.
\end{chunk}

\begin{chunk} \label{bakbi} Let $(R,\fm, k)$ be a local ring of prime characteristic $p$ and let $M \neq 0$ be a finitely generated $R$-module. We set: 
\begin{align*}
 \crsp(M) = & \min\{e\geq 0 \mid (0:_{M/ \x M)} \fm^{[p^e]}) \nsubseteq \fm^{[p^e]} (M/ \x M)  \text{ for an 
$M$-regular sequence } \x\}, \\
 \drsp(M) = & \min\{e\geq 0 \mid \fm^{[p^e]} M \subseteq (\underline{\mathbf{x}}) M \text{ for a system of parameters } \x \text{ on } M \}.
\end{align*}
It seems unknown whether or not $\sup\{\crsp(M_\fp) \mid \fp \in \Supp_R(M)\}$ is finite in general.
\end{chunk}

\begin{chunk} \label{bakbi2} Let $(R,\fm, k)$ be a local ring of prime characteristic $p$ and let $M \neq 0$ be a finitely generated Cohen-Macaulay $R$-module. Then $\crsp(M) \le \drsp(M)$. If $|k|= \infty$, then $\drsp(M) \le \lceil \log_p \e_R(M)\rceil$; see, for example, \cite[2.4]{CSY}.  Therefore, we have
\begin{align*}
\sup\big\{\crsp(M_\fp) \mid \fp \in \Supp_R(M)\big\} & \le
\begin{cases}\max\left\{ \lceil \log_p \e_R(M)\rceil,\, \crsp(M)\right\} \\
\lceil \log_p \e_R(M)\rceil \qquad \text{if }\ |k|= \infty
\end{cases}, \\
\sup\big\{\drsp(M_\fp) \mid \fp \in \Supp_R(M)\big\} &\le 
\begin{cases}\max\left\{ \lceil \log_p \e_R(M)\rceil,\, \drsp(M)\right\} \\
\lceil \log_p \e_R(M)\rceil \qquad \text{if }\ |k|= \infty
\end{cases}.
\end{align*}
\end{chunk}

\begin{chunk} \label{RS} Let $R \to S$ be a ring homomorphism, $M$ be an $S$-module, and let $X$ be an $R$-module. 
\begin{enumerate}[\rm(i)]
\item $\Hom_R(M, X)$ has an $S$-module structure as follows: For $s\in S$ and $\alpha \in \Hom_R(M, X)$, we define $s \cdot \alpha \in \Hom_R(M, X)$ as $(s \cdot \alpha)(m)=\alpha(sm)$ for all $m\in M$.
\item $\Hom_R(X, M)$ has an $S$-module structure as follows: For $s\in S$ and $\alpha \in \Hom_R(X, M)$, we define $s \cdot \alpha \in \Hom_R(X, M)$ as $(s \cdot \alpha)(n)=s \alpha(x)$ for all $x\in X$.
\item $M \otimes_R X$ has an $S$-module structure as follows: For $s\in S$ and $\sum_i m_i \otimes x_i \in M \otimes_R X$, with $m_i \in M$ and $x_i \in X$, we define $s \cdot \left (\sum_i m_i \otimes x_i \right)= \sum_i (sm_i) \otimes x_i \in M \otimes_R X$.
\end{enumerate}
\end{chunk}

The following observation is used in several proofs in the sequel. Note that, over a Noetherian ring, every finitely generated module is finitely presented.

\begin{chunk}\label{tor-ext-duality-general}\label{tor-ext-duality} Let $f: R \to S$ be a ring homomorphism, $A$ be an $R$-module, and let $B$ be an $S$-module. Assume $E$ is an injective $S$-module. Given $n\geq 0$, the following hold:
\begin{enumerate}[\rm(i)]
\item $\Hom_S\big(\Tor^R_n(A, B), E\big) \cong \Ext^n_R\big(A, \Hom_S(B, E)\big)$; see \cite[10.63]{Roit}.
\item If $A$ is finitely presented over $R$, then $\Hom_S\big(\Ext_R^n(A, B), E\big) \cong \Tor_n^R\big(A, \Hom_S(B, E)\big)$. 
\item If $(R, \fm, k)$ is local, $R=S$, and $f=\id$, then parts (i) and (ii) imply that:
\[\fd_R(B) = \id_R(B^\vee) \quad \text{ and } \quad  \id_R(B) = \fd_R(B^\vee).\]
\end{enumerate}
\end{chunk}

\begin{chunk} [{Auslander-Buchsbaum \cite{HD2}}] \label{AB2} Let $(R,\fm, k)$ be a $d$-dimensional local ring and let $N$ be an $R$-module. If $\id_R(N)<\infty$, then $\id_R(N)\leq d$. Thus, if $\fd_R(N)<\infty$, then $\fd_R(N)\leq d$; see~\ref{tor-ext-duality-general}(iii). 
\end{chunk}

The next observation is used in the proofs of Lemma~\ref{lem:d+1} and Proposition~\ref{prop:injective}.

\begin{chunk} \label{rmkinj} Let $(R,\fm, k)$ be a $d$-dimensional local ring and let $N$ be an $R$-module. 
Consider a minimal injective resolution of $N$:
$$\I=\ (0\xra{\ \ \ } I^0 \xra{\;h_1\;} I^1 \xra{\;h_2\;} \cdots \xra{\;h_n\;} I^n \xra{h_{n+1}} I^{n+1} \xra{\ \ \ } \cdots),$$
where $\displaystyle{I^j=\bigoplus_{\fp \in \Spec(R)}E_R(R/\fp)^{\oplus \mu_j(\fp, N)}}$ and $\mu_j(\fp, N)=\rank_{k(\fp)}\big(\Ext_{R_{\fp}}^j(k(\fp),N)\big)$ for $j\geq 0$. Here, $\mu_j(\fp, N)$ is not necessarily finite. Note that $\Hom_{R_\fp}(k(\fp), (h_j)_\fp) =0$ for $j \ge 0$. If $\fq \in \Spec(R)$ and we localize $\I$ at $\fq$, then the resulting $\I_{\fq}$ is a minimal injective resolution of $N_{\fq}$ over $R_{\fq}$ and $\displaystyle{I^j_{\fq}=\bigoplus_{\fp \subseteq \fq}E_R(R/\fp)^{\oplus \mu_j(\fp, N)}}$ for all $j\geq 0$; see, for example, \cite{BassU} and \cite[3.15 and Appendix 20-24]{LCbook}.

Assume $\id_{R_{\fp}}(N_{\fp})<\infty$ for all $\fp \in \Spec(R)-\{\fm\}$. Then $\id_{R_{\fp}}(N_{\fp})\leq \dim(R_{\fp})\leq d-1$ for all $\fp \in \Spec(R)-\{\fm\}$; see \ref{AB2}. This implies that  $I^i =  E_R(k)^{\oplus \mu_i(\fm, N)}$ for all $i \geq d$. Consequently, $\I$ has the following form:
$$\I=\ (0\to I^0 \to \cdots \to I^{d-1} \to E_R(k)^{\oplus \mu_d(\fm, N)} \xra{h_{d+1}} E_R(k)^{\oplus \mu_{d+1}(\fm, N)} \to \cdots).$$
\end{chunk}

\begin{chunk} \label{hom-tensor=0} Let $(R,\fm,k)$ be a local ring, $I$ be an ideal of $R$ such that $\sqrt{I}=\sqrt{\Ann_R(M)}$, $M$ be a finitely generated $R$-module, and let $N$ be an $R$-module. Then,
\begin{enumerate}[\rm(i)]
\item $\Hom_R(M,N)=0 \iff \gr_R(I, N) \geq 1 \iff \Hom_R(R/I,N) = 0$; see \cite[1.2.3 or 1.2.10(e)]{BH}.
\item It follows from part (i) and \ref{tor-ext-duality}(i) that 
\begin{align}\notag{}
\notag{} M \otimes_R N = 0 & \iff  \Hom_R(M, N^\vee) = 0  \iff   \Hom_R(R/I, N^\vee) = 0   \\ & \notag{} \iff  (R/I \otimes_R N)^\vee = 0 \iff R/I \otimes_R N = 0 \iff N=I N. 
\end{align}
\item If $\Supp_R(M) = \Spec(R)$ (or equivalently, $I=0$), it follows from parts (i) and (ii) that $$\Hom_R(M,N)=0 \iff M \otimes_R N = 0 \iff N = 0.$$
\item $\Ass_R(\Hom_R(M,N)) = \Supp_R(M) \cap \Ass_R(N)$; see, for example, \cite[1.2.28]{BH}.
\end{enumerate}
\end{chunk}

\begin{rmk} Assume $I$, $M$, and $N$ are as in \ref{hom-tensor=0}, but $R$ is (Noetherian as always) not necessarily local. Then the implications considered in \ref{hom-tensor=0} still hold since they can all be verified locally. More precisely, $\Hom_R(M,N)=0  \iff  \Hom_R(R/I,N) = 0  \text{ \;\;and\;\; }   M \otimes_R N = 0   \iff  N=IN.$
Also, we have that: $\Supp_R(M) = \Spec(R) \Longrightarrow \Hom_R(M,N)=0 \iff  N = 0   \iff  M \otimes_R N = 0$.
These implications establish the fact $\Ass_R\big(\Hom_R(M,N)\big) = \Supp_R(M) \cap \Ass_R(N)$, namely the equality stated in \ref{hom-tensor=0}(iv), still holds.
\end{rmk}

\begin{chunk}\label{AuBrsequence}  Let $(R, \fm, k)$ be a local ring and let $M$ be a finitely generated $R$-module. Consider a minimal free representation $P_1 \xra{\partial_1} P_0 \to M \to 0$. The \emph{transpose} $\Tr_R M$ of $M$ is defined as the cokernel of the $R$-dual map $\partial_1^{\ast}=\Hom_{R}(\partial_1,R)$. We refer the reader to \cite{AuBr} for the details of the following:
\begin{enumerate}[\rm(i)]
\item There is an exact sequence of $R$-modules $0\rightarrow M^*\rightarrow P_0^*\rightarrow P_1^*\rightarrow \Tr_R M\rightarrow 0$.
\item It follows that, up to isomorphism, $\Tr_R M$ is uniquely determined. 
\item It follows that, up to projectives, $\Tr_R (\Tr_R M) \cong M$.
\item $M$ is free if and only if $\Tr_R M$ is free.
\end{enumerate}
\end{chunk}

The following result from \cite{CSY} is necessary for our proof of Theorem~\ref{CSY-gen-intro}. 

\begin{chunk} [{\cite[2.2]{CSY}}] \label{CSY0} Let $(R,\fm, k)$  be a local ring and let $M$ and $N$ be finitely generated $R$-modules.
Assume $n\geq 1$ is an integer. Assume further the following conditions hold:
\begin{enumerate}[\rm(i)]
\item $N_{\fp}$ is free for all $\fp \in \Spec(R)-\{\fm\}$.
\item $\depth_R(M\otimes_RN)\geq n$.
\item $\depth_R(M)\geq n-1$.
\end{enumerate}
Then $\Ext^i_R(\Tr_R N,M)=0$ for all $i=1, \ldots, n$.
\end{chunk}


\section{A generalization of a result of Funk-Marley}\label{sec:Frob-inj}

Let us start by recalling the result of Funk-Marley stated in the introduction; see \ref{FM-intro}.

\begin{chunk} [{Funk-Marley \cite[3.1 and 3.2]{FM}}] \label{FM-section3} Let $(R,\fm, k)$ be a $d$-dimensional Cohen-Macaulay local ring of prime characteristic $p$, with $d\geq 1$, and let $N$ be an $R$-module. Given integers $t\geq 1$ and $e \gg 0$, the following hold: 
\begin{enumerate}[\rm(i)]
\item If $\Tor_i^R(\up{e}R, N)=0$ for all $i=t,\, \dots,\,
t+r-1$, then $\fd_R(N) \le d$. 
\item If $R$ is F-finite and $\Ext^i_R(\up{e}R, N)=0$ for all $i=t,\, \dots,\,
t+r-1$, then $\id_R(N)\leq d$. 
\end{enumerate}
\end{chunk}

The original statement of \ref{FM-section3} in \cite{FM} includes the case where $d=0$; in fact this case follows from the techniques of Koh-Lee used in the proof of \cite[2.6]{KL}; see also \cite[2.8]{FM} and \cite[2.2.8]{Millercont}. Note also that \cite[3.1 and 3.2]{FM}, namely \ref{FM-section3}, is stated in terms of complexes of $R$-modules, but its proof naturally reduces to the case of modules. For this reason, we consider only modules when generalizing \ref{FM-section3} in Theorems~\ref{injective} and ~\ref{p1}, which can also be extended to the complex case in a similar manner, as explained in the proof of \cite[3.1]{FM}.

The main results of this section are captured in the following theorem, featured in the introduction as parts (i) and (ii) of Theorem~\ref{main-thm-intro}. 

\begin{thm} \label{injective} 
Let $(R,\fm, k)$ be a $d$-dimensional local ring of prime characteristic $p$, $M$ be a finitely generated Cohen-Macaulay $R$-module such that $\Supp_R(M) = \Spec(R)$, and let $N$ be an $R$-module. Assume $\depth(R) \geq 1$. Given $t\geq 1$ and $e \gg 0$, the following hold: 
\begin{enumerate}[\rm(i)]
\item If $\Tor^R_i(\up{e}M, N)=0$ for all $i=t,\, \dots,\,t+d-1$, then $\fd_R(N) \leq d$. 
\item If $R$ is F-finite and $\Ext^i_R(\up{e}M, N)=0$ for all $i=t,\, \dots,\,t+d-1$, then $\id_R(N)\leq d$. 
\end{enumerate}
\end{thm}

\begin{rmk} Let $R$ be a ring of prime characteristic $p$, $M$ be a finitely generated $R$-module such that $\Supp_R(M) = \Spec(R)$, and let $N$ be an $R$-module. Let $e \ge 0$. 
\begin{enumerate}[\rm(i)]
\item Assume $R$ is F-finite. Then $\Hom_R(\up{e}M, N) =0 \implies N = 0$; see~\ref{hom-tensor=0}(iii). 
\item Without the F-finite assumption, we have $\up{e}M \otimes N = 0 \implies N = 0$, as locally this reduces to the F-finite case and then follows from~\ref{hom-tensor=0}(iii). 
\end{enumerate}
Thus, Theorem~\ref{injective} still holds when $t = 0$. This relies on $\Supp_R(\up{e}M) = \Spec(R)$ and does not depend on the choice of $e$ or $M$ being Cohen-Macaulay.
\end{rmk}

Theorem~\ref{injective} generalizes \ref{FM-section3} in the case where $\depth(R) \ge 1$. For a generalization of \ref{FM-section3} in the case where $\depth(R)=0$ (that is, the $d=0$ case of \ref{FM-section3}), see Proposition~\ref{d+1}. 
Before presenting our \hyperref[proof:injective]{proof of Theorem~\ref*{injective}} at the end of the section, we would like to discuss the sharpness of the result and list some corollaries of the theorem. We will also prove Propositions~\ref{d+1} and \ref{prop:injective}, which the \hyperref[proof:injective]{proof of Theorem~\ref*{injective}} relies on.

The following example shows that the positive depth assumption on the ring is necessary for Theorem~\ref{injective}. 

\begin{eg}\label{ex:ext-2} Let $R=\mathbb{F}_p[\![x,y]\!]/(x^2,xy)$ and let $M=R/(x)$. Then $R$ is F-finite, $\depth(R)=0$, $\dim(R) = 1$, and $M$ is a Cohen-Macaulay $R$-module such that $\Supp_R(M)=\Spec(R)$. Moreover, $\up{e}M \cong M^{\oplus p^e}$ for all $e\geq 0$.

Let $N=R/(y)$. Then $\Tor_1^R(M,N)=0$ so that $\Tor_1^R(\up{e}M,,N)=0$ for all $e\geq 0$. If $\pd_R(N)<\infty$, then $N$ is free since $\depth(R)=0$. Hence $\pd_R(N)=\fd_R(N)=\infty$. This shows that the positive depth assumption is needed for Theorem~\ref{injective}(i).

Next let $N=M$. Then $\Ext^1_R(M,N)=0$ so that $\Ext^1_R(\up{e}M, M)=0$ for all $e\geq 0$. If $\id_R(N)<\infty$, then $\id_R(N)=\depth(R)=0$, that is, $N$ is injective. Hence, $\id_R(N)=\infty$ (one can also conclude that $\id_R(N)=\infty$ since $R$ is not Cohen-Macaulay). This shows that the positive depth assumption is needed for Theorem~\ref{injective}(ii).
\end{eg}

We give several corollaries of Theorem~\ref{injective}. The next corollary covers the particular case where $S = \wh R$, the $\fm$-adic completion of $(R,\fm)$.

\begin{cor}\label{cor:injective} Let $f: (R,\fm,k) \to (S,\fn,\ell)$ be a flat local ring homomorphism, where $R$ is a $d$-dimensional local ring of prime characteristic $p$ and $\fm S = \fn$,
$M$ be a finitely generated Cohen-Macaulay $R$-module such that $\Supp_R(M)=\Spec(R)$, and let $N$ be a finitely generated $S$-module. Assume $\depth(R)\geq 1$. Given $t\geq 1$ and $e \gg 0$, if $\Tor_i^R(\up{e}M, N)=0$ for all $i = t, \dotsc, t+d-1$, then $\pd_{S}(N)\leq d$.
\end{cor}

\begin{proof} Observe $\Tor_{i}^R(k,N) \cong \Tor_{i}^S(\ell,N)$ for all $i\geq 0$. By Theorem~\ref{injective}, we have that $\fd_R(N) \leq d$.  
Set $r=\fd_R(N)$. It follows that $\Tor_{r}^R(k,N)\neq 0=\Tor_{r+1}^R(k,N)$. Therefore, $\pd_{S}(N) = r\leq d$.
\end{proof}

\begin{cor} \label{torla} Let $(R,\fm, k)$ be a $d$-dimensional local ring of prime characteristic $p$, $M$ be a finitely generated Cohen-Macaulay $R$-module such that $\Supp_R(M) = \Spec(R)$, and let $N$ be an $R$-module. Assume $\depth(R)\geq 1$. Given $t\geq 1$ and $e \gg 0$, if $\Ext_R^i(\up{e}M, N^\vee) =0$ for all $i=t,\, \dots,\,t+d-1$, then $\id_R(N^\vee) \leq d$. 
\end{cor}

\begin{proof} 
The vanishing of $\Ext_R^i(\up{e}M, N^\vee) $  yields the vanishing of $\Tor^R_i(\up{e}M, N)$; see \ref{tor-ext-duality}(i). Thus $\id_R(N^\vee) = \fd_R(N) \leq d$ by Theorem~\ref{injective}(i) and \ref{tor-ext-duality}(iii).
\end{proof}

\begin{cor}\label{cor:dual-injective} Let $(R,\fm, k)$ be a $d$-dimensional local ring of prime characteristic $p$, $M$ be a finitely generated Cohen-Macaulay $R$-module such that $\Supp_R(M)=\Spec(R)$, and let $N$ be an $R$-module. Assume $\depth(R)\geq 1$. Given $t\geq 1$ and $e \gg 0$, assume at least one of the following conditions holds:
\begin{enumerate}[\rm(i)]
\item $\Ext^i_R\Big(N,\, \big(\up{e}M\big)^\vee\Big)=0$ for all $i = t, \dotsc, t+d-1$.
\item $\Ext^i_R\Big(N,\, \up{e}\big(M^\vee\big)\Big)=0$ for all $i = t, \dotsc, t+d-1$. 
\end{enumerate}
Then $\fd_R(N) \leq d$.
\end{cor}

\begin{proof} 
The vanishing of $\Ext^i_R\big(N, (\up{e}M)^\vee\big)$ yields the vanishing of $\Tor_i^R\big(N, \up{e}M\big)$; see \ref{tor-ext-duality}(i). Hence case~(i) follows from Theorem~\ref{injective}(i).

Similarly, the vanishing of $\Ext^i_R\big(N,\, \up{e}(M^\vee)\big)$ yields the vanishing of $\Tor_i^R\big(N, \up{e}M\big)$; see \ref{tor-ext-duality-general}(i). Therefore, case~(ii) also follows from Theorem~\ref{injective}(i).
\end{proof}

Next, we prepare some auxiliary results for the \hyperref[proof:injective]{proof of Theorem~\ref*{injective}}. To begin with, we present Lemma~\ref{lem:d+1} and Corollary~\ref{cor:d+1} which are akin to \cite[4.5 and 4.6]{FM}.

\begin{lem} \label{lem:d+1}  Let $(R,\fm, k)$ be a $d$-dimensional F-finite local ring of prime characteristic $p$, $M$ be a finitely generated $R$-module such that $\Supp_R(M) = \Spec(R)$, and let $N$ be an $R$-module. Set $\delta =\max\{\depth_{R_\fp}(M_\fp) \mid \fp \in \Spec(R)\}$ and let $t\geq 1$ be an integer. If
\[
\sup\{e \mid \text{$\Ext^i_R(\up{e}M, N)=0$ for all $i=t,\, \dots,\,t+\delta$}\} \ge \sup\{\crsp(M_\fp) \mid \fp \in \Spec(R)\},
\]
then $\id_R(N)\leq d$. In particular, if $\Ext^i_R(\up{e}M, N)=0$ for all $i=t,\, \dots,\,t+\delta$ and for infinitely many $e$, then $\id_R(N)\leq d$.
\end{lem}

\begin{proof} Set $d=\dim(R) = d$. We proceed by induction on $d$ to show that $\id_R(N)<\infty$. 
The case where $d=0$ follows from \ref{DS}(2)(iii). Hence, we assume $d\geq 1$. 

As $R$ is F-finite, $\up{e}M$ is a finitely presented $R$-module and hence $\Ext^i_R(\up{e}M, N)_{\fp} \cong \Ext^i_{R_{\fp}}(\up{e}M_{\fp}, N_{\fp})$ for all $\fp \in \Spec(R)$. So, for all $\fp \in \Spec(R)-\{\fm\}$, the induction hypothesis dictates $\id_{R_{\fp}}(N_{\fp})< \infty$. According to \ref{rmkinj}, $N$ has a minimal injective resolution of the form
$$\I=\ (0\to I^0 \to \cdots \to I^{d-1} \to E_R(k)^{\oplus \mu_d} \to \cdots \to E_R(k)^{\oplus \mu_{t+d}}  \to \cdots),$$
in which $\mu_i=\mu_i(\fm, N)$ for $i \ge d$. Say $\depth_R(M) = v$. 
By assumption, there exists $e \ge \crsp(M)$ such that $\Ext^i_R(\up{e}M, N)=0$ for $i=t,\, \dots,\,t+v$.
Since $e \ge \crsp(M)$, there exists a maximal $M$-regular sequence $\x= \{x_1, \ldots, x_v\}$ such that $\big(0:_{M/ (\x) M} \fm^{[p^e]}\big) \nsubseteq \fm^{[p^e]} \big(M/ (\x) M\big)$; see \ref{bakbi}. Thus, $k$ is a direct summand of $\up{e}\big(M/(\x) M\big)$; see \ref{DS}.

As $x_1$ is $M$-regular, there is an short exact sequence $0 \to \up{e}M \to \up{e}M \to \up{e}(M/x_1M) \to 0$. This, together with $\Ext^i_R(\up{e}M,N)=0$ for $i=t,\, \dots,\,t+v$, implies that 
\[\Ext^i_R(\up{e}(M/x_1M),N)=0 \quad \text{for \  all } \  i=t+1,\, \dots,\,t+v. \]
As $\x = \{x_1, \dots,x_v\}$ is $M$-regular, inductively we get $\Ext^{t+d}_R\big(\up{e}(M/(\x)M),N\big)=0$. 
Therefore, $\Ext^{t+v}_R(k, N) = 0$ since $k$ is a direct summand of $\up{e}(M/(\x) M)$. In view of~\ref{rmkinj}, we deduce
\[\mu_{t+v}=\rank_k(\Ext^{t+v}_R(k, N))= 0.\] 
Thus $I^{t+v}= E_R(k)^{\oplus \mu_{t+v}} = 0$ and $\id_R(N)<t+v$. Consequently, by \ref{AB2}, we see that $\id_R(N) \le d$.
\end{proof}

The next result is a corollary of Lemma~\ref{lem:d+1}.

\begin{cor} \label{cor:d+1}  Let $(R,\fm, k)$ be a $d$-dimensional local ring of prime characteristic $p$, $M$ be a finitely generated $R$-module such that $\Supp_R(M) = \Spec(R)$, and let $N$ be an $R$-module. Let $S$ be a ring extension of $R$ as in \ref{finite}. Set $\delta=\max\{\depth_{S_\fp}(M \otimes_R S_\fp) \mid \fp \in \Spec(S)\}$ and let $t\geq 1$ be an integer. If
\[
\sup\{e \mid \text{$\Tor_i^R(\up{e}M, N)=0$ for all $i=t,\, \dots,\,t+\delta$}\} \ge \sup\{\crsp(M \otimes S_\fp) \mid \fp \in \Spec(S)\},
\]
then $\fd_R(N)\leq d$. In particular, if $\Tor_i^R(\up{e}M, N)=0$ for all $i=t,\, \dots,\,t+\delta$ and for infinitely many $e$, then $\fd_R(N)\leq d$.
\end{cor}

\begin{proof} It suffices to show $\fd_S(N \otimes_R S) \le \dim(S)$. Thus, given the assumption, we may assume that $R = S$ is already F-finite. 
By \ref{tor-ext-duality}(i), the vanishing of $\Tor_i^R(\up{e}M, N)$ yields vanishing of $\Ext^i_R(\up{e}M, N^\vee)$.
Now the claim follows from Lemma~\ref{lem:d+1}; see \ref{tor-ext-duality}(iii).
\end{proof}

As an application of Lemma~\ref{lem:d+1} and Corollary~\ref{cor:d+1}, we obtain the following proposition that generalizes a result of Takahashi-Yoshino \cite[5.3]{RY}. 
 
\begin{prop} \label{d+1} Let $(R,\fm, k)$ be a $d$-dimensional local ring of prime characteristic $p$, $M$ be a finitely generated Cohen-Macaulay $R$-module such that $\Supp_R(M) = \Spec(R)$, and let $N$ be an $R$-module. Then, given $t\geq 1$ and $e \gg 0$, the following hold:
\begin{enumerate}[\rm(i)]
\item If $\Tor_i^R(\up{e}M, N)=0$ for all $i=t,\, \dots,\,t+d$, then $\fd_R(N)\leq d$. 
\item If $R$ is F-finite and $\Ext^i_R(\up{e}M, N)=0$ for all $i=t,\, \dots,\,t+d$, then $\id_R(N)\leq d$. 
\end{enumerate}
\end{prop}

\begin{proof} For (i), we may assume that $R$ is F-finite with $|k| = \infty$; see \ref{finite}. In this case, we have $\sup\{\crsp(M_\fp) \mid \fp \in \Spec(R)\} \le  \lceil \log_p \e_R(M)\rceil< \infty$. The rest follows from Corollary~\ref{cor:d+1}.

For (ii), without the assumption that $|k| = \infty$, we have the following:
$$\sup\{\crsp(M_\fp) \mid \fp \in \Spec(R)\} \le \max\{ \lceil \log_p \e_R(M)\rceil,\, \crsp(M)\} < \infty.$$
The rest follows from Lemma~\ref{lem:d+1}. 
\end{proof}

Our proof of Proposition~\ref{prop:injective} is inspired by some of the techniques employed by Funk-Marley in the proof of \cite[3.2]{FM}. To distinguish various module structures in the proof, we present \ref{prop:injective} in the context of a general ring homomorphism $f: R \to S$. Subsequently, in the \hyperref[proof:injective]{proof of Theorem~\ref*{injective}}, we apply Proposition~\ref{prop:injective} to $F^e: R \to R$, the iterated Frobenius endomorphism. Recall that, over a local ring $(R,\fm,k)$, we set $(-)^\vee=\Hom_R\big(-,\,E_R(k)\big)$. 

\begin{prop} \label{prop:injective} Let $f:(R,\fm,k)\to (S,\fn, \ell)$ be a module-finite local homomorphism of local rings, with $d=\dim(R) \ge 1$, $M$ be an $S$-module, and let $N$ be an $R$-module. Assume the following: 
\begin{enumerate}[\rm(i)]
\item $\Ext^i_R(M, N)=0$ for all $i=t,\, \dots,\,t+d-1$ for some $t\geq 1$. 
\item $M$ is a finitely generated $S$-module such that $\Supp_R(M) = \Spec(R)$. 
\item There exists $\x = \{x_1, \,\dotsc,\,x_d \} \subseteq \fn$ such that $\x$ is $M$-regular and $\fm M \subseteq (\x)M$.
\item $\id_{R_\fp}(N_\fp) < \infty$ for all $\fp \in \Spec(R)-\{\fm\}$. 
\item $\fd_{R_\fp}\big((N^{\vee})_\fp\big) < \infty$ for all $\fp \in \Ass(R)$. 
\end{enumerate}
Then $\id_R(N)\leq d$. 
\end{prop} 

\begin{proof} 
As $\id_{R_\fp}(N_\fp) < \infty$ for each $\fp \in \Spec(R)-\{\fm\}$, \ref{rmkinj} dictates that $N$ has a minimal injective resolution of the following form:
$$\I=\ (0\to I^0 \xra{h_{1}} \cdots  \to I^{d-1} \xra{h_{d}}  \E_R(k)^{\oplus \mu_{d}(\fm, N)} \to \cdots \xra{h_{t+d}} \E_R(k)^{\oplus \mu_{t+d}(\fm, N)} \to  \cdots).$$
We apply $\Hom_R(M,-)$ to $\I$ and use our assumption~(i) to obtain an exact sequence:
\begin{equation}\label{prop:injective-seq1}\tag{\ref*{prop:injective}.1}
M_{t-1} \xra{\ g_t\ } 
\dotsb \xra{g_{t+d-2}} M_{t+d-2} \xra{g_{t+d-1}} M_{t+d-1} \xra{\ g_{t+d}\ } M_{t+d} \xra{\ \rho \ }
C \xra{\ \ \ } 0, 
\end{equation}
where $M_i= \Hom_R(M,I^i)$, $g_i=\Hom_R(M,h_i)$ and $C=\coker(g_{t+d})$. Note that each $M_i$ is an $S$-module; see~\ref{RS}(i). Hence, $C$ is an $S$-module as well. \smallskip

\noindent \textbf{Claim~1.} \phantomsection \label{claim1}
The induced sequence $\ol{M_{t+d-1}} \xra{\ \ol{g_{t+d}}\ } \ol{M_{t+d}} 
\xra{\ \ol{\rho}\ } \ol{C} \xra{\ \ \ } 0$ is exact, in which
\[
\ol {M_i}= \Hom_R(M/(\x)M,I_i), \ \   \ol{g_{t+d}}= \Hom_R(M/(\x)M,h_{t+d}) \;\; \text{ and } \;\; \ol {C}= \Hom_S(S/(\x),C).
\]
\noindent \emph{Proof of Claim~1}. As $I^i$ is injective over $R$ and $x_1$ is regular on $M$, an application of $\Hom_R(-,I^i)$ to 
the exact sequence $0 \to M \xra{x_1} M \xra{\ \ \ } M/x_1M \to 0$ induces 
an short exact sequence: 
\[\tag{$\Lambda_i$}\label{prop:injective-seq-i}
0 \lla M_i \xla{\ x_1 \ } M_i \lla \Hom_R(M/x_1M, I^i) \lla 0.
\]
We combine the short exact sequences \eqref{prop:injective-seq-i}, for all $i = t-1,\dotsc, t+d$,
with the exact sequence \eqref{prop:injective-seq1}, and obtain the following exact sequence of $S$-modules:
\[
\Hom_R(M/x_1M,I^t) \to \dotsb \to \Hom_R(M/x_1M,I^{t+d}) \to \Hom_S(S/x_1S,C) \to 0.
\]
Inductively, as $\x = \{x_1,\dotsc,x_d\}$ is an $M$-regular sequence, 
we 
realize the following exact sequence of $S$-modules that is naturally induced from \eqref{prop:injective-seq1}, as claimed:
\[
\Hom_R(M/(\x)M,I^{t+d-1}) \xra{\ol{g_{t+d}}}  
\Hom_R(M/(\x)M,I^{t+d}) \xra{\ \ol{\rho}\ } \Hom_S(S/(\x)S,C) \to 0.
\]
\noindent \textbf{Claim~2.} \phantomsection \label{claim2} $\ker(\rho) = 0$. 
\smallskip

\noindent \emph{Proof of Claim~2}. The assumption $\fm M \subseteq (\x)M$ says that, as an $R$-module, $M/(\x)M$ is a direct sum of $k$. Since $\I$ is a minimal injective resolution of $N$, we see that $$\ol{g_{t+d}}= \Hom_R(M/(\x)M, h_{t+d}) =0.$$ 
In view of the exact sequence in \hyperref[claim1]{Claim~1}, we get $\ker(\ol{\rho}) = 0$. Moreover, $\ol{\rho}$ can be identified as $\Hom_S(S/(\x), \rho)$ up to the natural isomorphism $\Hom_S(S/(\x),M_i) \cong \Hom_R(M/(\x)M,I^i)$. Hence $(0:_{\ker(\rho)} (\x)) \cong \ker(\ol{\rho}) = 0$. Given that $M$ is finitely generated over $R$, we have $$\Ass_R(\ker(\rho)) \subseteq \Ass_R(M_{t+d}) \subseteq \Ass_R(\E_R(k)^{\oplus \mu_{d+t}(\fm, N)}) \subseteq \{\fm\}.$$ Thus $\Ass_S(\ker(\rho)) \subseteq \{\fn\}$, as $\fn$ is the only prime ideal of $S$ lying over $\fm$. So, if $\ker(\rho) \neq 0$, then $\fn \in \Ass_S(\ker(\rho))$ and hence $0 \neq (0:_{\ker(\rho)} \fn) \subseteq (0:_{\ker(\rho)} (\x))$, which contradicts the conclusion $(0:_{\ker(\rho)} (\x)) = 0$ above. This completes the proof of \hyperref[claim2]{Claim~2}. \smallskip

Now that we know $\ker(\rho) = 0$, the exact sequence \eqref{prop:injective-seq1} forces $g_{t+d}=0$, which gives rise to an exact sequence as follows:
\[
\Hom_R(M,I^{t+d-2})\xra{\ g_{t+d-1}\ } \Hom_R(M,I^{t+d-1})  
\xra{\ \ \quad \ \ } 0.
\]
Since $M$ is finitely presented over $R$, we apply $(-)^\vee$ to the exact sequence above and obtain the following exact sequence in light of~\ref{tor-ext-duality}(ii):
\[
M \otimes_R (I^{t+d-2})^{\vee} \xla{\ 1 \otimes h_{t+d-1}^{\vee}\ } 
M \otimes_R (I^{t+d-1})^{\vee} \xla{\ \ \quad \ \ } 0.
\tag{\ref*{prop:injective}.2} \label{prop:injective-seq2}
\]

Next, let us return to the injective resolution $\I$ of $N$. Consider the exact sequence 
\[
0 \xra{\ \ \ } N \xra{\ \ \ } I^0 \xra{\ h_{1}\ } \dotsb \xra{\ \ \ \ } 
I^{t+d-2} \xra{h_{t+d-1}} I^{t+d-1} \xra{\ \theta \ } D \xra{\ \ \ } 0,
\]
where $D$ is the cokernel of the map $h_{t+d-1}$. 
Applying $(-)^\vee$ to the exact sequence above, we get an exact sequence:
\[
0 \xla{\ \ \ } N^{\vee} \xla{\ \ \ } (I^0)^{\vee} \xla{\ \ \ } \dotsb 
\xla{\ \ \ } (I^{t+d-2})^{\vee} \xla{h_{t+d-1}^{\vee}} (I^{t+d-1})^{\vee} 
\xla{\ \theta^{\vee} \ } D^{\vee} \xla{\ \ \ } 0 
\tag{\ref*{prop:injective}.3} \label{prop:injective-seq3}
\]
with each $(I^{i})^{\vee}$ flat over $R$; see \ref{tor-ext-duality}(iii). In particular, we see
\[
\Ass_R(D^{\vee}) \subseteq \Ass_R((I^{t+d-1})^{\vee}) \subseteq \Ass(R).
\]
\noindent \textbf{Claim~3.} \phantomsection \label{claim3}
$(D^{\vee})_\fp = 0$ for all $\fp \in \Ass(R)$. 
\smallskip

\noindent \emph{Proof of Claim~3}. 
Fix any $\fp \in \Ass(R)$. From assumption~(v), we see $\fd_{R_\fp}((N^{\vee})_\fp)
\le \dim(R_\fp) \le d$; see \ref{AB2}. Thus, \eqref{prop:injective-seq3} localized at
$\fp$ gives rise to a flat resolution of $N_\fp$ over $R_\fp$, which can be used to compute $\Tor_{i}^{R_\fp}(-, (N^{\vee})_\fp)$. 
By \ref{tor-ext-duality}(ii), we have
\[
\Tor_{t+d-1}^{R_\fp}\big(M_\fp, (N^{\vee})_\fp\big) =\big (\Tor_{t+d-1}^{R}(M, N^{\vee})\big)_\fp
\cong \big(\Ext^{t+d-1}_R(M, N)^\vee\big)_\fp =0.
\] 
Moreover, we have $\Tor_{t+d}^{R_\fp}(M_\fp, (N^{\vee})_\fp)=0$ since $t+d >\fd_{R_\fp}((N^{\vee})_\fp)$.
Next, we apply $M_\fp \otimes_{R_\fp}-$ to \eqref{prop:injective-seq3} localized at $\fp$. The vanishing of 
$\Tor_{i}^{R_\fp}(M_\fp, (N^{\vee})_\fp)$, for $i = t+d-1$ and $i = t+d$,
forces the following exact sequence:
\[
M_\fp \otimes_{R_\fp} ((I^{t+d-2})^{\vee})_\fp
\xla{1 \otimes (h_{t+d-1}^{\vee})_\fp}
M_\fp \otimes_{R_\fp} ((I^{t+d-1})^{\vee})_\fp
\xla{1 \otimes (\theta^{\vee})_\fp} M_\fp \otimes_{R_\fp} (D^{\vee})_\fp \xla{\ \ \ } 0.
\]
Comparing this with \eqref{prop:injective-seq2} localized at $\fp$, we see
$M_\fp \otimes_{R_\fp} (D^{\vee})_\fp = 0$. Hence $(D^{\vee})_\fp = 0$ since $\Supp_{R_\fp}(M_\fp) = \Spec(R_\fp)$;
see~\ref{hom-tensor=0}. This completes the proof of \hyperref[claim3]{Claim~3}. \smallskip

Overall, we have $\Ass_R(D^{\vee}) \subseteq \Ass(R)$ and $(D^{\vee})_\fp = 0$ for all $\fp \in \Ass(R)$. This forces $D^{\vee} = 0$, which implies $D = 0$. Therefore, $\id_R(N) < t+d$, so $\id_R(N) \le d$ by \ref{AB2}. 
\end{proof}

We are now ready to prove Theorem~\ref{injective}, which is a consequence of Propositions~\ref{d+1} and~\ref{prop:injective}. Recall that $F:R\to R$ denotes the Frobenius endomorphism. 

\begin{proof}[Proof of Theorem~\ref{injective}]  \phantomsection \label{proof:injective}
It suffices to prove part (i) for the case where $R$ is F-finite with $|k| = \infty$; see ~\ref{finite}. Thus, we can obtain part~(i) from part~(ii) via duality; see \ref{tor-ext-duality}(iii). 

To prove part (ii), note that $\up{e}M$ is a finitely generated (and hence a finitely presented) $R$-module with $\Supp_R(\up{e}M) = \Supp_R(M)= \Spec(R)$. By our choice of $e$, there exists an $M$-regular sequence $\x = \{x_1, \,\dotsc,\,x_d \} \subseteq \fm$ such that $\fm (\up{e}M) \subseteq \up{e}\big((\x)M\big)$; see \ref{bakbi}. 
In light of \ref{tor-ext-duality}(ii), the vanishing of $\Ext^i_R(\up{e}M, N)$ implies the vanishing of $\Tor_i^R(\up{e}M, N^\vee)$ for all $i=t,\, \dots,\,t+d-1$. Therefore, $\id_{R_\fp}(N_\fp) < \infty$ and $\fd_{R_\fp}\big((N^{\vee})_\fp\big) < \infty$ for all $\fp \in \Spec(R)-\{\fm\}$; see Proposition~\ref{d+1}. 
Note that $\Ass(R) \subseteq \Spec(R)-\{\fm\}$ since $\depth(R)\geq 1$. 
Now we apply Proposition~\ref{prop:injective}, with $R=S$ and $f=F^e$, and deduce that $\id_R(N)\leq d$.   
\end{proof}

\begin{rmk} \label{bounds-on-e} We conclude this section by pointing out some lower bounds for the integer 
$e$ that ensure the validity of the proofs of certain previously stated results.
\begin{enumerate}[\rm(a)]
\item In part (i) of both Theorem~\ref{injective} and Proposition \ref{d+1}, it is enough to assume $e \ge \lceil \log_p \e_R(M)\rceil$. 
\item In part (ii) of both Theorem~\ref{injective} and Proposition \ref{d+1}, it is enough to assume $e \ge \lceil \log_p \e_R(M)\rceil$ if $|k|=\infty$. Also, if $|k|<\infty$, it is enough to assume $e \ge \max\left\{\lceil \log_p \e_R(M)\rceil,\, \drsp(M)\right\}$.

\item In Corollaries~\ref{cor:injective}, \ref{torla}, and \ref{cor:dual-injective}, it is enough to assume $e \ge \lceil \log_p \e_R(M)\rceil$.
\end{enumerate}
\end{rmk}

It is proved in \cite[2.17]{MelYongwei} that, if $R$ is an excellent ring and $M$ is a finitely generated $R$-module, then the set $\sup\left\{ \e_{R_{\fm}}(M_{\fm}): \fm \in \Max(R) \right\}$ is finite; see \cite[2.17]{MelYongwei}  We use this fact and state a global version of Theorem~\ref{injective}.
 
\begin{thm} \label{injective-global} 
Let $R$ be a $d$-dimensional ring of prime characteristic $p$, with $d\geq 1$, $M$ be a finitely generated Cohen-Macaulay $R$-module such that $\Supp_R(M) = \Spec(R)$, and let $N$ be an $R$-module. Assume $\depth(R_{\fm}) \geq \min \big\{1, \dim(R_{\fm}) \big\}$ for each maximal ideal $\fm$ of $R$. 
\begin{enumerate}[\rm(i)]
\item Assume $R$ is excellent, $s=\sup\left\{ \e_{R_{\fm}}(M_{\fm}): \fm \in \Max(R) \right\}$, and let $e$ be an integer such that $e\geq \lceil \log_p s \rceil$. Given $t\geq 1$, if $\Tor^R_i(\up{e}M, N)=0$ for all $i=t,\, \dots,\,t+d-1$, then $\fd_R(N) \leq d$. 
\item Assume $R$ is F-finite and the residue field of $R_{\fm}$ is infinite for each maximal ideal $\fm$ of $R$. Let $e$ be an integer such that $e\geq \sup\left\{\lceil \log_p \e_{R_{\fm}}(M_{\fm})\rceil,\, \drsp(M_{\fm}): \fm \in \Max(R) \right\}$. Given $t\geq 1$, if $\Ext^i_R(\up{e}M, N)=0$ for all $i=t,\, \dots,\,t+d-1$, then $\id_R(N)\leq d$. 
\end{enumerate}
\end{thm}

\begin{proof} Note that $s<\infty$ due to \cite[2.17]{MelYongwei}. In proving the first part, we have that $\Tor^{R_{\fm}}_i(\up{e}M_{\fm}, N_{\fm})=0$ for all $i=t,\, \dots,\,t+d-1$. If $\dim(R_{\fm})=0$, then the residue field of $R_{\fm}$ is a direct summand of $\up{e}M_{\fm}$, and hence $\fd_{R_{\fm}}(N_{\fm}) \leq \dim(R_{\fm})\leq d$; see \ref{DS}(2)(i). On the other hand, if  $\dim(R_{\fm})\geq 1$, then part (i) of Theorem~\ref{injective} yields $\fd_{R_{\fm}}(N_{\fm}) \leq d$; see Remark~\ref{bounds-on-e}(a). This implies that $\fd_R(N) \leq d$ as flat dimension can be computed locally. We can prove the second part similarly since $\sup\left\{\lceil \log_p \e_{R_{\fm}}(M_{\fm})\rceil,\, \drsp(M_{\fm}): \fm \in \Max(R) \right\}$ is finite under our setting; see Remark~\ref{bounds-on-e}(a). 
\end{proof}


\section{On the homological properties of the Frobenius endomorphism}\label{sec:Frob-proj}

The aim of this section is to prove Theorem~\ref{main-thm-intro}(iii) and also obtain Proposition~\ref{CSY-gen-intro}. Recall that Proposition~\ref{CSY-gen-intro} generalizes \cite[1.3]{CSY}, recalled as \ref{CSY} in the introduction.

The layout of this section is as follows: We begin by preparing some auxiliary results. Then we establish Theorem~\ref{main-thm-intro}(iii) in Theorem~\ref{p1}. Finally, making use of \ref{CSY0} and Theorem \ref{p1}, we produce a \hyperref[proof:CSY-gen-intro]{proof of Proposition~\ref*{CSY-gen-intro}}. Along the way, we also discuss the sharpness of our results; see Examples \ref{ornek1} and \ref{ornek2}. 

The first auxiliary result, namely Proposition \ref{p0}, is akin to \cite[2.6]{KL} and \cite[Theorem~A]{NTY}. Even though it suffices to apply the proposition to the identity map $1_R: R \to R$ in the sequel, we present it more generally in terms of a ring homomorphism $f: R \to S$.

\begin{prop} \label{p0}  Let $f: (R,\fm, k) \to (S, \fn, \ell)$ be a local homomorphism of local rings of prime characteristic $p$ with $d=\dim(R)$, $M \neq 0$ be a finitely generated $R$-module with $v=\depth_R(M)$, and let $N$ be a finitely generated $S$-module. Given $t \ge 1$ and $e \geq  \crsp(M)$, if
$\Ext^i_R(N,\up{e}M)=0$ for all $i=t,\, \dots,\,t+v$, then $\fd_R(N) \le t-1$ and hence $\pd_R(N)\leq t-1+d$.
\end{prop}

\begin{proof} By our choice of $e$, there exists a maximal $M$-regular sequence $\x = \{x_1, \dots,x_v\}$ such that $k$ is a direct summand of $\up{e}(M/(\x)M)$ over $R$; see \ref{DS} and \ref{bakbi}. 

As $x_1$ is $M$-regular, there is an short exact sequence $0 \to \up{e}M \to \up{e}M \to \up{e}(M/x_1M) \to 0$. This, together with $\Ext^i_R(N,\up{e}M)=0$ for $i=t,\, \dots,\,t+v$, implies that 
\[\Ext^i_R(N,\up{e}(M/x_1M))=0\quad \text{for all} \quad i=t,\, \dots,\,t+v-1. \]
Inductively, as $\x = \{x_1, \dots,x_v\}$ is $M$-regular, we get $\Ext^t_R(N,\up{e}(M/(\x)M))=0$. Since $k$ is a direct summand of $\up{e}(M/(\x)M)$, we see that $\Ext^t_R(N,k)=0$,
which implies $\fd_R(N) \le t-1$; see \cite[2.1]{NTY}. Note that every flat $R$-module has projective dimension at most $d$; see \cite[4.2.8]{FlatCover}. Therefore, we conclude that $\pd_R(N)\leq t-1+d$.
\end{proof}

As in Propositions~\ref{prop:injective} and ~\ref{p0}, we present Proposition~\ref{prop1} in the context of a general ring homomorphism $f: R \to S$, allowing us to distinguish various module structures in the proof. When Proposition~\ref{prop1} is applied in the proof of Theorem~\ref{p1}, the homomorphism will be $F^e: R \to R$, the $e$-th iteration of the Frobenius endomorphism.

\begin{prop}\label{prop1} Let $(R,\fm,k)$ be a local ring, $f: R \to S$ be a ring homomorphism, $N$ be a finitely generated $R$-module, and let $M$ be a finitely generated $S$-module. Assume the following hold:
\begin{enumerate}[\rm(i)]
\item $\Ext^i_R(N,M)=0$ for all $i=t,\, \dots,\,t+d-1$ for some $d\geq 1$ and $t\geq 1$. 
\item There exists $\x = \{x_1, \dotsc, x_d\} \subseteq \Jac(S)$ such that $\x$ is $M$-regular and $\fm M \subseteq (\x)M$. 
\end{enumerate}
Then $\Hom_R\big(\Omega^t_R(N),M\big) = 0$. 
\end{prop}

\begin{proof} 
Consider a minimal free resolution of $N$ over $R$:
\[
\F= \ (\dotsb \xra{\ \ \ } F_{t+d} \xra{h_{t+d}} F_{t+d-1} \xra{h_{t+d-1}} \dotsb
\xra{h_{t+1}} F_{t} \xra{\ h_{t}\ } F_{t-1} \xra{\ \ \ } \dotsb \xra{\ \ \ } 0).
\]
Applying $\Hom_R(-,M)$ to $\F$ and using assumption~(i),
we get an induced exact sequence: 
\[\label{prop1-seq}\tag{\ref*{prop1}.1}
M_{t+d} \xla{g_{t+d}} M_{t+d-1} \xla{g_{t+d-1}} \dotsb
\xla{g_{t+1}} M_{t} \xla{\ g_{t}\ } M_{t-1}
\xla{\ \iota\ } G \xla{\ \ \ } 0,
\]
in which $M_i= \Hom_R(F_i,M)$, $g_i= \Hom_R(h_i,M)$, 
and $G= \ker(g_t)$.   
All $M_i$, and hence $G$, are $S$-modules; see~\ref{RS}(ii).
In fact, each $M_i$ is isomorphic to a finite direct sum of $M$. Thus all $M_i$ are finitely generated $S$-modules.

As $F_i$ is free over $R$ and $x_1$ is regular on $M$, an application of $\Hom_R(F_i,-)$ to 
the exact sequence $0 \to M \xra{\ x_1 \ } M \xra{\ \ \ } M/x_1M \to 0$ induces 
an short exact sequence:
\[ \tag{$\Gamma_i$}\label{prop1-seq-i}
0 \lra M_i \xra{\ x_1 \ } M_i \lra \Hom_R(F_i,M/x_1M) \lra 0.
\]
We combine the exact sequences \eqref{prop1-seq-i}, for all $i = t-1,\dotsc, t+d$, 
with the exact sequence \eqref{prop1-seq} and obtain the following exact sequence of $S$-modules:
\[
\Hom_R(F_{t+d-1},M/x_1M) \la \dotsb \la \Hom_R(F_{t-1},M/x_1M) \la G/x_1G \la 0.
\]
Inductively, as $\x = \{x_1,\dotsc,x_d\}$ is an $M$-regular sequence, 
we realize the following exact sequence of $S$-modules that is naturally induced from \eqref{prop1-seq}:
\[\label{prop1-seq2}\tag{\ref*{prop1}.2}
\ol{M_{t}} \xla{\ \ \ol{g_{t}}\ \ } \ol{M_{t-1}} \xla{\ \ \ol{\iota}\ \ } \ol{G} \xla{\ \ \ \ \ } 0,
\]
in which $\ol {M_i} = \Hom_R(F_i,M/(\x)M)$, $\ol{g_t}= \Hom_R(h_t,M/(\x)M)$ and $\ol G= G/(\x)G$. 
Up to isomorphism, we may write $\ol {M_i} = M_i/(\x)M_i= M_i \otimes_S S/(\x)S$, $\ol G = G \otimes_S S/(\x)S$ and thus $\ol{\iota} = \iota \otimes 1_{S/(\x)S}$. 

The assumption $\fm M \subseteq (\x)M$ implies that $M/(\x)M$ is annihilated by $\fm$. 
Since $\F$ is a minimal free resolution of $N$, we conclude $\ol{g_{t}} = \Hom_R(h_t, M/(\x)M) = 0$. Thus, the exactness of \eqref{prop1-seq2} forces $\im(\ol{\iota}) = \ker(\ol{g_t}) =\ol{M_{t-1}}$, meaning that $\im(\iota) + (\x)M_{t-1} = M_{t-1}$. As $M_{t-1}$ is finitely generated over $S$ and $(\x)S \subseteq \Jac(S)$, 
we obtain $\im(\iota) = M_{t-1}$, thanks to Nakayama's lemma. This forces $\ker(g_{t+1}) = \im(g_t) = 0$ due to the exactness of \eqref{prop1-seq}.
Finally, as $\Omega^t_R(N) \cong \coker(h_{t+1})$, we see
\[
\Hom_R\big(\Omega^t_R(N),M\big) \cong \ker\big(\Hom_R(h_{t+1},M)\big) = \ker(g_{t+1}) = 0. \qedhere
\]
\end{proof}

Equipped with Propositions~\ref{p0} and \ref{prop1}, we are now ready to prove the result stated in Theorem~\ref{main-thm-intro}(iii).

\begin{thm}\label{p1} Let $(R,\fm, k)$ be a $d$-dimensional local ring of prime characteristic $p$, $M$ be a finitely generated Cohen-Macaulay $R$-module such that $\Supp_R(M)=\Spec(R)$, and let $N$ be a finitely generated $R$-module. Assume $\depth(R)\geq 1$. Given $t\geq 1$ and $e \gg 0$, if $\Ext^i_R(N,\up{e}M)=0$ for all $i=t,\, \dots,\,t+d-1$, then $\pd_R(N)\leq \min\{t-1,\,\depth(R)\}$. 
\end{thm}

\begin{proof} Note that, by our choice of $e$, there exists $\x = \{x_1, \,\dotsc,\,x_d \} \subseteq \fm$ such that $\x$ is $M$-regular and $\fm (\up{e}M) \subseteq \up{e}\big((\x)M\big)$; see \ref{bakbi}. Upon an application of Proposition~\ref{prop1} to $F^e: R \to R$, the vanishing of $\Ext^i_R(N,\up{e}M)$ for $i=t,\, \dots,\,t+d-1$ implies $\Hom_{R}\big(\Omega^t_R(N), \up{e}M\big) = 0$.

To prove the claim by contradiction, suppose $\Omega_R^t(N)\neq 0$ and pick $\fp \in \Ass_R(\Omega_R^t(N))\subseteq \Ass(R)$. Note that $\fp \neq \fm$ since $\depth(R)\ge 1$. It follows from Proposition~\ref{p0} that $\pd_{R_{\fp}}(N_{\fp})<\infty$. Hence, $N_{\fp}$ is free by the Auslander-Buchsbaum formula. 
So,  $\Omega^t_R(N)_\fp \neq 0$ is free over $R_\fp$. Also, since $\Supp_R(\up{e}M) = \Supp_R(M)=\Spec(R)$, we see that $(\up{e}M)_\fp \neq 0$. 
Hence $\Hom_{R_\fp}(\Omega^t_R(N)_\fp, (\up{e}M)_\fp) \neq 0$, which contradicts the conclusion $\Hom_{R}\big(\Omega^t_R(N), \up{e}M\big) = 0$.
Thus $\Omega_R^t(N)=0$, so $\pd_R(N) \le t-1$. This proves that $\pd_R(N)\leq \min\{t-1,\,\depth(R)\}$.
\end{proof}

We now record several corollaries of Theorem \ref{p1}. It is worth noting that the positive depth assumption in the theorem is necessary; see Example \ref{ex:ext-2}.

\begin{cor}\label{cor:dual-p1} Let $(R,\fm, k)$ be a $d$-dimensional local ring of prime characteristic $p$, $M$ be a finitely generated Cohen-Macaulay $R$-module such that $\Supp_R(M)=\Spec(R)$, and let $N$ be a finitely generated $R$-module. Assume $\depth(R)\geq 1$. Given $t\geq 1$ and $e \gg 0$, assume that at least one of the following conditions holds:
\begin{enumerate}[\rm(i)]
\item $\Tor_i^R\Big(N, \big(\up{e}M\big)^\vee\Big)=0$ for all $i = t, \dotsc, t+d-1$.
\item 
$\Tor_i^R\Big(N, \up{e}\big(M^\vee\big)\Big)=0$ for all $i = t, \dotsc, t+d-1$. 
\end{enumerate}
Then $\pd_R(N) \leq \min\{t-1,\,\depth(R)\}$.
\end{cor}

\begin{proof} 
The vanishing of $\Tor_i^R\big(N, (\up{e}M)^\vee\big)$ yields the vanishing of $\Ext^i_R\big(N, \up{e}M\big)$; see \ref{tor-ext-duality}(ii). Hence case~(i) follows from Theorem~\ref{p1}. Similarly, the vanishing of $\Tor_i^R\big(N,\, \up{e}(M^\vee)\big)$ yields the vanishing of $\Ext^i_R\big(N, \up{e}M\big)$; see \ref{tor-ext-duality-general}(ii). Therefore, case~(ii) also follows from Theorem~\ref{p1}.
\end{proof}

Recall that $F^e_R(-)$ denotes the scalar extension along $F^e: R \to R$, the $e$-th iteration of the Frobenius endomorphism. Given $R$-modules $X$ and $Y$, it follows that $X \otimes_R \up{e}Y \cong \up{e}(F^e_R(X) \otimes_R Y)$ and $F^e_{R_\fp}(X_\fp) \cong (F^e_R(X))_\fp$ for all $\fp \in \Spec(R)$. Moreover, $\Supp_R(X) = \Supp_R(F^e_R(X))$ in case $X$ is a finitely $R$-module.

Next, we provide a proof of Proposition~\ref{CSY-gen-intro} and restate it here the convenience of the reader.

\begin{cor} \label{CSY-gen} Let $(R,\fm, k)$ be a local ring of prime characteristic $p$, and let $M$ and $N$ be finitely generated $R$-modules. Assume the following conditions hold:
\begin{enumerate}[\rm(i)]
\item $M$ is Cohen-Macaulay and $\Supp_R(M) = \Spec(R)$. 
\item $N$ is generically free, that is, $N_{\fp}$ is a free $R_{\fp}$-module for all $\fp \in \Ass(R)$. 
\end{enumerate}
If $F_R^e(N)\otimes_RM$ is a maximal Cohen-Macaulay $R$-module for some $e \gg 0$, then $N$ is free.  
\end{cor}

\begin{proof} 
\phantomsection\label{proof:CSY-gen-intro} Without loss of generality we may assume $R$ is F-finite and $|k| = \infty$; see \ref{finite}. Set $d=\dim(R)$ and proceed by induction on $d$.  

As $R$ is F-finite, the assumption that $F_R^e(N)\otimes_RM$ is maximal Cohen-Macaulay can be interpreted as that the module $N \otimes_R \up{e}M$ is maximal Cohen-Macaulay. If $\depth(R)=0$, then $N$ is free since we assume it is generically free. Hence, we assume $\depth(R)\geq 1$. Note that, by the induction hypothesis, we may assume $N_{\fp}$ is a free  $R_\fp$-module for all $ \fp \in \Spec(R) - \{\fm\}$. Note also that $\depth_R(\up{e}M)=d = \depth_R(N \otimes_R \up{e}M)$. 
Therefore, we use \ref{CSY0} for the case where $n=d$ and conclude that $\Ext^i_R(\Tr_R N, \up{e}M)=0$ for all $i=1, \ldots, d$. Now Theorem~\ref{p1} implies that $\Tr_R N$ is free. Consequently, $N$ is free; see \ref{AuBrsequence}(iv).
\end{proof}

\begin{cor} \label{min1} Let $(R,\fm, k)$ be a local ring of prime characteristic $p$, and let $M$ and $N$ be finitely generated $R$-modules. Assume the following conditions hold:
\begin{enumerate}[\rm(i)]
\item $M$ is Cohen-Macaulay.
\item $M$ has constant rank on $\Min(R)$ and $N$ is generically free.
\end{enumerate}
If $F_R^e(N)\otimes_RM$ is maximal Cohen-Macaulay for some $e \gg 0$, then $N$ is free. 
\end{cor}

\begin{proof} Suppose $\Min(R) \nsubseteq \Supp_R(M)$. Then $\Supp_R(M) \cap \Min(R)=\emptyset$ since $M$ has constant rank on $\Min(R)$. This implies that $\dim_R(M)<\dim(R)$. However, this is not possible as $\dim(R)=\dim_R\big(F_R^e(N)\otimes_RM\big)\leq \dim_R(M)$. So, $\Min(R) \subseteq \Supp_R(M)$ and hence $\Supp_R(M)=\Spec(R)$. Now, the claim follows from Theorem \ref{CSY-gen}.
\end{proof}

\begin{cor} \label{min2} Let $(R,\fm, k)$ be a local ring of prime characteristic $p$, $M$ be a finitely generated Cohen-Macaulay $R$-module, and let $N$ be a finitely generated $R$-module. Assume:
\begin{enumerate}[\rm(i)]
\item For each $\fp \in \Min(R)$, there is an integer $r_{\fp}\geq 1$ such that $M_{\fp}\cong R_{\fp}^{\oplus r_{\fp}}$. 
\item $N$ is generically free and $\dim_R(N)=\dim(R)$.
\end{enumerate}
If $F^e_R(N)\otimes_R M$ is Cohen-Macaulay for some $e \gg 0$, then $N$ is free.
\end{cor}

\begin{proof} It follows that $\dim_R\big(F_R^e(N)\otimes_RM\big)=\dim_R\big(F_R^e(N)\big)=\dim_R(N)=\dim(R)$ because $\Supp_R(M)=\Spec(R)$.
Now, the claim follows from Theorem \ref{CSY-gen}.
\end{proof}

We illustrate the sharpness of Theorem~\ref{CSY-gen-intro} with some examples. Specifically, Example~\ref{ornek1} demonstrates that the theorem's conclusion may fail if the module $M$ in question does not have full support. Similarly, Example~\ref{ornek2} highlights the necessity of the assumption that $M$ is generically free for the theorem to hold.

\begin{eg}\label{ornek1} Let $R=\mathbb{F}_p[\![x,y]\!]/(xy)$, $M=R/(x)$, and $N=M$. Note that $M$ is maximal Cohen-Macaulay, and $N_{\fp}$ is a free $R_{\fp}$-module for all $\fp \in \Ass(R)$.  It follows that $N\otimes_R\up{e}M$ is maximal Cohen-Macaulay for all $e\geq 0$ since
$N\otimes_R\up{e}M \cong N\otimes_RM^{\oplus p^e}\cong M^{\oplus p^e}$. However, $N$ is not free. Here we have that $\Spec(R) \ni (y) \notin \Supp_R(M)$.
\end{eg}

\begin{eg}\label{ornek2} Let $R=\mathbb{F}_p[\![x,y]\!]/(x^2,xy)$, $M=R/(x)$, and $N=M$. Note that $M$ is Cohen-Macaulay, $\Supp_R(M)=\Spec(R)$, and $N_{\fp}$ is a free $R_{\fp}$-module for all $\fp \in \Min(R)$. It follows that $N\otimes_R\up{e}M$ is maximal Cohen-Macaulay for all $e\geq 0$ since $N\otimes_R\up{e}M \cong M^{\oplus p^e}$. However, $N$ is not free. Here we have that $ \Ass(R) \ni \fm \notin \Min(R)$. 
\end{eg}

It is proved in \cite[1.1]{AHIY} that, if there is a finitely generated module $N \neq 0$ over a local ring $R$ of prime characteristic $p$ such that $\fd_R(\up{r}N) < \infty$ or $\id_R(\up{r}N) < \infty$ for some ${r} \ge 1$, then $R$ is regular. We use this fact and obtain the following consequences of Theorem~\ref{main-thm-intro}; cf. \cite[6.8 and 6.10]{Test2}.

\begin{cor} \label{cor:main-thm-intro} Let $(R,\fm, k)$ be a $d$-dimensional local ring of prime characteristic $p$, $M$ be a finitely generated Cohen-Macaulay $R$-module such that $\Supp_R(M) = \Spec(R)$, and let $N \neq 0$ be a finitely generated $R$-module. Assume $\depth(R) \geq 1$. Given $t\geq 1$, ${r} \ge 1$ and $e \gg 0$, we further assume that at least one of the following conditions holds:
\begin{enumerate}[\rm(i)]
\item $\Tor^R_i(\up{e}M, \up{r}N)=0$ for all $i=t,\, \dots,\,t+d-1$.
\item $R$ is F-finite and $\Ext^i_R(\up{e}M, \up{r}N)=0$ for all $i=t,\, \dots,\,t+d-1$.
\item $\up{r}N$ is finitely generated over $R$ and $\Ext^i_R(\up{r}N, \up{e}M)=0$ for all $i=t,\, \dots,\,t+d-1$. 
\end{enumerate}
Then $R$ is regular.
\end{cor}

\begin{proof} In view of Theorem~\ref{main-thm-intro}, the claims follow from \cite[1.1]{AHIY}.
\end{proof}

\begin{rmk} We conclude this section by pointing out some lower bounds for the integer 
$e$ that ensure the validity of the proofs of certain previously stated results.
\begin{enumerate}[\rm(a)]
\item In both Corollary~\ref{CSY} and Theorem~\ref{CSY-gen}, it is enough to assume $e \ge \lceil \log_p \e_R(M)\rceil$.
\item In both Theorem~\ref{p1} and Corollary \ref{cor:dual-p1}, it is enough to assume $e \ge \lceil \log_p \e_R(M)\rceil$. This is because, in their proofs, we may use \ref{finite} and assume $|k| = \infty$.
\item In Corollary~\ref{cor:main-thm-intro}, it is enough to assume $e \ge \max\left\{\lceil \log_p \e_R(M)\rceil,\, \drsp(M)\right\}$. Also, if $|k| = \infty$, then it is enough to assume $e \ge \lceil \log_p \e_R(M)\rceil$ in Corollary~\ref{cor:main-thm-intro}.
\end{enumerate}
\end{rmk}

We do not know whether or not the Cohen-Macaulay assumption on $M$ is necessary in Theorem~\ref{CSY-gen-intro}. This raises the following question:

\begin{ques} Let $(R,\fm, k)$ be a local ring of prime characteristic $p$ and let $M$ and $N$ be finitely generated $R$-modules. Assume 
$\Supp_R(M) = \Spec(R)$ and $N$ is generically free. If $F^e_R(N)\otimes_R M$ is maximal Cohen-Macaulay for some $e \gg 0$, then must $N$ be free? What if $R$ is Cohen-Macaulay, or $N=M$?
\end{ques} 

We complete this section with a global version of Theorem~\ref{p1}; we skip its proof as it is similar to that of Theorem~\ref{injective-global}.

\begin{thm}\label{p1-global} Let $R$ be a $d$-dimensional excellent ring of prime characteristic $p$, $M$ be a finitely generated Cohen-Macaulay $R$-module such that $\Supp_R(M)=\Spec(R)$, and let $N$ be a finitely generated $R$-module. Assume $\depth(R_{\fm}) \geq \min \big\{1, \dim(R_{\fm}) \big\}$ for each maximal ideal $\fm$ of $R$. Set $s=\sup\left\{ \e_{R_{\fm}}(M_{\fm}): \fm \in \Max(R) \right\}$, and let $e$ be an integer such that $e\geq \lceil \log_p s \rceil$. Given $t\geq 1$, if $\Ext^i_R(N,\up{e}M)=0$ for all $i=t,\, \dots,\,t+d-1$, then $\pd_R(N)\leq \min\{t-1,\,\depth(R)\}$. 
\end{thm}

\newpage

\section*{Acknowledgements}
Part of this work was completed during Celikbas's visit to the Department of Mathematics and Statistics at Georgia State University in February 2018, and during the visit of Celikbas, Sadeghi, and Yao to the Nesin Mathematics Village in May 2019. The authors thank these two institutions for their hospitality. They also thank Mohsen Asgarzadeh for his comments on a previous version of the manuscript.

\end{document}